\documentclass[12pt]{amsart}
\usepackage{amssymb,amsmath}
\usepackage{comment}
\usepackage{tikz}
\usetikzlibrary{matrix,arrows,backgrounds,shapes.misc,shapes.geometric,patterns,calc,positioning}

\usepackage[colorlinks=true, pdfstartview=FitV, linkcolor=blue, citecolor=blue, urlcolor=blue]{hyperref}

\usepackage{fullpage}
\usepackage{graphicx}
\usepackage{enumerate}
\def\VR{\kern-\arraycolsep\strut\vrule &\kern-\arraycolsep}
\def\vr{\kern-\arraycolsep & \kern-\arraycolsep}

\newtheorem{theorem}{Theorem}
\newtheorem*{theoremnonum}{Theorem}
\newtheorem{lemma}[theorem]{Lemma}
\newtheorem{prop}[theorem]{Proposition}
\newtheorem{corollary}[theorem]{Corollary}
\newtheorem{conjecture}[theorem]{Conjecture}

\theoremstyle{definition}
\newtheorem{definition}[theorem]{Definition}

\newtheorem{example}[theorem]{Example}

\newtheorem{question}[theorem]{Question}

\theoremstyle{remark}
\newtheorem{remark}[theorem]{Remark}

\newcommand{\Hom}{\operatorname{Hom}}
\newcommand{\End}{\operatorname{End}}
\newcommand{\Aut}{\operatorname{Aut}}
\newcommand{\Ext}{\operatorname{Ext}}

\newcommand{\Proj}{\operatorname{Proj}}
\newcommand{\SI}{\operatorname{SI}}
\newcommand{\SL}{\operatorname{SL}}
\newcommand{\GL}{\operatorname{GL}}
\newcommand{\PGL}{\operatorname{PGL}}

\newcommand{\ZZ}{\mathbb Z}

\newcommand{\NN}{\mathbb N}

\newcommand{\PP}{\mathbb P}

\newcommand{\Id}{\operatorname{Id}}

\newcommand{\Mat}{\operatorname{Mat}}

\newcommand{\rad}{\operatorname{rad}}

\newcommand{\ddim}{\operatorname{\mathbf{dim}}}
\newcommand{\dd}{\operatorname{\mathbf{d}}}

\newcommand{\hh}{\operatorname{\mathbf{h}}}

\newcommand{\M}{\operatorname{\mathcal{M}}}

\newcommand{\module}{\operatorname{mod}}

\newcommand{\rk}{\operatorname{rank}}

\newcommand{\key}[1]{\emph{#1}}

\begin{document}
\title{Module varieties and representation type of finite-dimensional algebras}
\author{Calin Chindris}
\address{University of Missouri-Columbia, Mathematics Department, Columbia, MO, USA}
\email[Calin Chindris]{chindrisc@missouri.edu}

\author{Ryan Kinser}
\address{Northeastern University, Department of Mathematics, Boston, MA, USA}
\email[Ryan Kinser]{r.kinser@neu.edu}

\author{Jerzy Weyman}
\address{Northeastern University, Department of Mathematics, Boston, MA, USA}
\email[Jerzy Weyman]{j.weyman@neu.edu}
\thanks{The first author was supported by NSF grant DMS-1101383, the second author by NSA Young Investigator Grant H98230-12-1-0244, and the third author by NSF grant DMS-0901885.}

\begin{abstract} In this paper we seek geometric and invariant-theoretic characterizations of (Schur-)representation finite algebras. To this end, we introduce two classes of finite-dimensional algebras:  those with the dense-orbit property and those with the multiplicity-free property. We show first that when a connected algebra $A$ admits a preprojective component, each of these properties is equivalent to $A$ being representation-finite. Next, we give an example of an algebra which is not representation-finite but still has the dense-orbit property.  We also show that the string algebras with the dense orbit-property are precisely the representation-finite ones. Finally, we show that a tame algebra has the multiplicity-free property if and only if it is Schur-representation-finite.  
\end{abstract}

\maketitle
\setcounter{tocdepth}{1}
\tableofcontents

\section{Introduction}
Throughout the article, $K$ always denotes an algebraically closed field of characteristic zero, and ``algebra'' refers to an associative $K$-algebra with identity. (We will remark at times on particular results that hold in arbitrary characteristic.)  All modules are assumed to be finite-dimensional left modules. We use interchangeably the vocabulary of modules over finite-dimensional algebras, and that of representations of quivers with relations.  A summary of the background on these, and their varieties of modules, is given in Section \ref{sect:background}.

\subsection{Motivation and context}
The representations of an algebra $A$ can be studied geometrically by considering the affine varieties $\module(A,\dd)$ of modules of fixed dimension vector, under the actions of the corresponding products of general linear groups $\GL(\dd)$.  In this setting, isomorphism classes of representations are precisely orbits, so the standard tools of algebraic geometry for determining orbits of a group acting on a variety become relevant.
We are interested in classifying those algebras whose module varieties satisfy certain invariant-theoretic properties, and to compare such classification results with the classical  representation theoretic properties of these algebras, in particular, the notion of finite representation type.
In this paper, we introduce the following properties for an algebra $A$: 
\begin{description}
\item[DO property]
for each dimension vector $\dd$ of $A$, $\GL(\dd)$ acts on each irreducible component of $\module(A,\dd)$ with a dense orbit;
\item[MF property]
for each dimension vector $\dd$ of $A$ and each irreducible component $C$ of $\module(A,\dd)$, the algebra of semi-invariants $K[C]^{\SL(\dd)}$ is multiplicity-free. 
\end{description}

Dense orbit properties similar to our DO property have been studied in other contexts; we briefly mention here some cases known to us, certainly this list must not be complete.  Sato and Kimura considered the situation of a linear algebraic group acting on a vector space \cite{MR0430336}; 
Richardson showed that a parabolic subgroup of a connected, semi-simple algebraic group has a dense orbit under the adjoint action on the Lie algebra of its unipotent radical \cite{MR0330311};  
Hille and R\"ohrle studied a generalization of Richardson's theorem involving the descending central series of a parabolic \cite{MR1987340};
and Hille and Goodwin studied a variation for a Borel subgroup of $\GL(n)$ \cite{MR2356319}.

Multiplicity-free actions have been intensively studied in many contexts. Important results on multiplicity-free actions include: the Peter-Weyl Theorem; $\GL_m-\GL_n$-duality and, more generally, Howe dualities \cite{MR986027}; branching laws for the general linear and orthogonal groups \cite{MR986027}; the classification of the multiplicity-free linear actions of reductive groups due to Kac \cite{MR575790} and Benson-Ratcliff \cite{MR1382030}; the classification of homogeneous spherical varieties due to Kramer \cite{MR528837} and Brion \cite{MR822838}. 

\subsection{Summary of results}
We first show that for algebras in a certain class studied by representation theorists -- namely, algebras with a preprojective component -- the two properties above are each equivalent to the algebra being finite representation type (see Theorem \ref{thm-preproj-comp} and the paragraph preceding it).

\begin{theoremnonum} Let $A$ be a connected, bound quiver algebra with a prepojective component. Then, the following properties are equivalent:
\begin{enumerate}
\renewcommand{\theenumi}{\arabic{enumi}}
\item $A$ is representation-finite;
\item $A$ has the DO property;
\item $A$ has the MF property.
\end{enumerate}
\end{theoremnonum}

It is clear that (1) implies (2) for any algebra, and we show in Proposition \ref{prop:schurfin-mf} that (2) implies (3) in general.
One might imagine these properties to be equivalent for all algebras.  There do not appear to be any  examples or theory in the literature demonstrating the existence of representation-infinite algebras with the DO or MF properties.  So we prove that the family of representation-infinite algebras below has the DO property to demonstrate that it is actually a novel concept (Theorem \ref{thm-example-DO}).  Our proof technique is an algorithm that yields an explicit list of the indecomposable representations with dense orbits \eqref{eq:indecomp}.

\begin{theoremnonum} Let $\Lambda$ be the algebra given by the following quiver and relations:
\begin{equation}
\vcenter{\hbox{
\begin{tikzpicture}[point/.style={shape=circle,fill=black,scale=.5pt,outer sep=3pt},>=latex]
\node[point,label={below:$1$}] (1) at (0,0) {} ;
\node[point,label={below:$2$}] (2) at (2,0) {} edge[in=45,out=-45,loop] node[right] {$b$} ();
\path[->] (1) edge node[above] {$a$} (2) ;
\end{tikzpicture} }}
\qquad
b^n = b^2 a = 0, \quad n\in \NN.
\end{equation}
Then $\Lambda$ is DO for all $n$, but infinite representation type for $n\geq 6$, and even wild representation type for $n > 6$.
\end{theoremnonum}

We should point out that the example above is of infinite global dimension; in fact, all our examples of representation-infinite DO algebras are 2-point algebras of infinite global dimension.  We believe that all representation-infinite algebras with the DO property should be non-triangular at least (Conjecture \ref{conj:repfiniteDO}).

It is easier to produce examples of representation-infinite algebras with the MF property, one such is given in Example \ref{ex:butterly}. In fact, we are able to characterize when a string algebra has the MF property in terms of a presentation by a quiver with relations (Corollary \ref{cor:MFstring}). 
We conjecture that the MF property should correspond to the representation theoretic notion of ``Schur-representation-finite'' in general (see Definition \ref{def:schurfin}). 
Finally, we are able to give representation theoretic characterizations of the DO and MF properties under certain conditions (Proposition \ref{string-not-DO} and Theorem \ref{thm:tamemf}):

\begin{theoremnonum}
Let $A$ be a bound quiver algebra.
\begin{enumerate}
\renewcommand{\theenumi}{\arabic{enumi}}
\item Assume that $A$ is a string algebra. Then, $A$ is representation-finite if and only if $A$ is DO.
\item Assume that $A$ is tame. Then, $A$ is Schur-representation-finite if and only if $A$ is MF.
\end{enumerate}
\end{theoremnonum}

\subsection*{Acknowledgements}  
We give special thanks  to Piotr Dowbor for his thoughts and comments over many months of working on the project.
We would also like to thank Raymundo Bautista, Harm Derksen, Christof Geiss, Lutz Hille, Birge Huisgen-Zimmermann, Jan Schr\"oer, and Dieter Vossieck for helpful conversations.  Finally, we thank  a referee whose comments lead to great improvement and simplification of the paper.

\section{Background}\label{sect:background}
\subsection{Module varieties}\label{sect:thebasics}
Up to Morita equivalence, any finite-dimensional, associative $K$-algebra $A$ can be viewed as a bound quiver algebra; that is, there exists a quiver $Q$ (uniquely determined by $A$) and an ideal $I$ in the path algebra $KQ$ such that $A \simeq KQ/I$.
Therefore, throughout the paper, we implicitly assume that our algebras are given by such a presentation. We say that $A$ is a \emph{triangular} algebra if $Q$ has no oriented cycles.  We refer to the text \cite{assemetal} for background on quivers and finite-dimensional algebras.

We write $Q_0$ for the set of vertices of a quiver $Q$, and $Q_1$ for its arrow set.  A dimension vector $\dd\colon Q_0 \to \NN$ for $A$ is a choice of a non-negative integer at each vertex of $Q$.
The affine \key{module variety} $\module(A,\dd)$ 
parametrizes the $A$-modules of dimension vector $\dd$ along with a fixed basis.  We represent a point of this variety by a collection of matrices associated to the arrows of $Q$ which satisfy the relations in $I$.  Writing $ta$ and $ha$ for the tail and head of an arrow $a$, we have
\begin{equation}
\module(A,\dd) = \{M \in \prod_{a \in Q_1} \Mat_{\dd(ha)\times \dd(ta)}(K) \mid M(r)=0, \forall r \in I\}.
\end{equation}
The orbits in $\module(A,\dd)$  of the base change group $\GL(\dd) = \prod \GL(\dd(i))$ are in one-to-one correspondence with the isomorphism classes of the $\dd$-dimensional $A$-modules.  See, for example, \cite{MR724715} for background on module varieties.

In general, $\module(A, \dd)$ does not have to be irreducible. Let $C$ be an irreducible component of $\module(A, \dd)$. We say that $C$ is \emph{indecomposable} if $C$ has a dense open subset of indecomposable modules.
As shown by de la Pe{\~n}a in \cite[\S1.3]{MR1113958} and Crawley-Boevey and Schr{\"o}er in \cite[Theorem~1.1]{MR1944812}, any irreducible component $C \subseteq \module(A, \dd)$ satisfies a Krull-Schmidt type decomposition
\begin{equation}
C=\overline{C_1\oplus \ldots \oplus C_t}
\end{equation}
for some indecomposable irreducible components $C_i\subseteq \module(A,\dd_i)$ with $\sum \dd_i = \dd$. We call $C=\overline{C_1\oplus \ldots \oplus C_t}$ \key{the generic decomposition of $C$}.

\begin{definition}
An algebra $A$ is said to have the \key{dense orbit property} (write $A$ is DO) if each irreducible component of each of its module varieties has a dense orbit.
\end{definition}

\begin{remark}\label{do-components} Using the generic decomposition, it is easy to see that an algebra is DO if and only if each of its indecomposable irreducible components has a dense orbit.
\end{remark}

\subsection{Dense orbits and self-extensions}\label{sect:doext}
The module varieties we study are the set of $K$-points of the module schemes $\underline{\module}(A, \dd)$, which only appear in the remarks of this subsection.
 The Artin-Voigt Lemma states that an $A$-module $M$ satisfies $\Ext^1_A(M,M)=0$ if and only if the orbit of $M$ is scheme-theoretically open in $\underline{\module}(A, \dd)$ \cite[II.3.5]{MR0486168}. This seems to be the only known representation-theoretic condition implying that a module has a dense orbit.

Note that if $A$ is representation finite, then it is trivially DO since each $\module(A,\dd)$ has only finitely many orbits.  To find non-trivial examples, one might look for algebras having an abundance of modules with self-extensions.  It would be interesting to know if there exist algebras having a module without self extensions in each component, which are not representation finite.
If one assumes that \emph{every} indecomposable $M$ satisfies $\Ext^1_A(M,M)=0$, then $A$ is already representation finite \cite[Lemma~4]{MR1106345}.  

Of course, vanishing self-extensions does not completely characterize modules with dense orbits, as one can see with an easy (even commutative) example: let $A=K[x]/(x^2)$, so the variety of modules of dimension 1 has only one point, thus a dense orbit.
However, the corresponding trivial module admits a nontrivial self extension:
\begin{equation}
0 \to K \to K[x]/(x^2) \to K \to 0.
\end{equation}
The issue is that if the scheme structure on a component of $\module(A, \dd)$ is not generically reduced, as in this example, then an orbit can be topologically open and dense without being scheme-theoretically open.

\subsection{Weight spaces of semi-invariants and moduli spaces of modules}\label{sect:moduli}
Consider the action of $\SL(\dd):=\prod_{i \in Q_0}\SL(\dd(i),K)$ on the module variety $\module(A,\dd)$, and the induced action on its coordinate ring.
The resulting ring of \emph{semi-invariants}  has a weight space decomposition over the group $X^\star(\GL(\dd))$ of rational characters of $\GL(\dd)$:
\begin{equation}
\SI(A,\dd):=K[\module(A,\dd)]^{\SL(\dd)} = \bigoplus_{\chi \in X^\star(\GL(\dd))}\SI(A,\dd)_{\chi}.
\end{equation}
Each of these summands 
\begin{equation}
\SI(A,\dd)_{\chi}=\lbrace f \in K[\module(A,\dd)] \mid g f= \chi(g)f \text{~for all~}g \in \GL(\dd)\rbrace
\end{equation} 
is a vector space called \key{the space of semi-invariants on} $\module(A,\dd)$ \key{of weight $\chi$}. Note that any $\theta \in \ZZ^{Q_0}$ defines a rational character $\chi_{\theta}:\GL(\dd) \to K^*$ by 
\begin{equation}
\chi_{\theta}((g(i))_{i \in Q_0})=\prod_{i \in Q_0}\det g(i)^{\theta(i)}.
\end{equation}
In this way, we identify $\ZZ ^{Q_0}$ with $X^\star(\GL(\dd))$, assuming that $\dd$ is a
sincere dimension vector. We also refer to the rational characters of $\GL(\dd)$ as (integral) weights of $A$ (or $Q$).

For an irreducible component $C \subseteq \module(A,\dd)$, we similarly define the ring of semi-invariants $\SI(C):=K[C]^{\SL(\dd)}$, and the space $\SI(C)_{\theta}$ of semi-invariants on $C$ of weight $\theta \in \ZZ^{Q_0}$.

\begin{definition}
An algebra $A$ is said to have the \key{multiplicity-free property} (write $A$ is MF) if the algebra of semi-invariants on each irreducible component $C$ of each of its module varieties is multiplicity-free; that is, $\dim_K \SI(C)_{\theta} \leq 1$ for all $\theta \in \ZZ^{Q_0}$.
\end{definition}

\begin{remark} Note that for an irreducible component $C \subseteq \module(A,\dd)$, the affine categorical quotient $C//\SL(\dd)$ is a $\GL(\dd)/\SL(\dd)$-variety and $\dim_K \SI(C)_{\theta}$ is the multiplicity of the 1-dimensional irreducible representations of the torus $\GL(\dd)/\SL(\dd)$ of weight $\theta$ in the coordinate ring $\SI(C)$ of $C//\SL(\dd)$. With this observation in mind, an algebra $A$ is MF if and only if for each dimension vector $\dd$ and irreducible component $C \subseteq \module(A,\dd)$, $C//\SL(\dd)$ is a multiplicity-free $\GL(\dd)/\SL(\dd)$-variety or, equivalently, $C//\SL(\dd)$ contains a dense $\GL(\dd)/\SL(\dd)$-orbit. 
\end{remark}

In our study of MF algebras, we will use of some of the main results on moduli spaces of representations due to King \cite{Kmodulireps}.  An $A$-module $M$ is said to be \emph{$\theta$-semi-stable} if $\theta(\ddim M)=0$ and $\theta(\ddim M')\leq 0$ for all submodules $M' \leq M$. We say that $M$ is \emph{$\theta$-stable} if $M$ is non-zero, $\theta(\ddim M)=0$, and $\theta(\ddim M')<0$ for all submodules $0 \neq M' < M$. Now, consider the (possibly empty) open subsets
\begin{equation}
\module(A,\dd)^{ss}_{\theta}=\{M \in \module(A,\dd)\mid M \text{~is~}
\text{$\theta$-semi-stable}\}
\end{equation}
and 
\begin{equation}
\module(A,\dd)^s_{\theta}=\{M \in \module(A,\dd)\mid M \text{~is~}
\text{$\theta$-stable}\}
\end{equation}
of $\dd$-dimensional $\theta$(-semi)-stable $A$-modules.

Using methods from Geometric Invariant Theory, King showed in \cite{Kmodulireps} that the projective variety
\begin{equation}
\M(A,\dd)^{ss}_{\theta}:=\Proj(\bigoplus_{n \geq 0}\SI(A,\dd)_{n\theta})
\end{equation}
is a GIT-quotient of $\module(A,\dd)^{ss}_{\theta}$ by the action of $\PGL(\dd)$. Here,  $\PGL(\dd)=\GL(\dd)/T_1$ where $T_1=\{(\lambda \Id_{\dd(i)})_{i \in Q_0} \mid \lambda \in K^*\} \leq \GL(\dd)$. Note that there is a well-defined action of $\PGL(\dd)$ on $\module(A,\dd)$ since $T_1$ acts trivially on $\module(A,\dd)$. We say that $\dd$ is a \emph{$\theta$-semi-stable dimension vector} if $\module(A,\dd)^{ss}_{\theta} \neq \emptyset$. 

It was proved in \cite[Proposition~4.2]{Kmodulireps} that the (closed) points of $\M(A,\dd)^{ss}_{\theta}$ are in one-to-one correspondence with the isomorphism classes of those modules in $\module(A,\dd)^{ss}_{\theta}$ that can be written as direct sums of $\theta$-stable modules. We call such $A$-modules \key{$\theta$-polystable}.

For an irreducible component $C \subseteq \module(A,\dd)$, we similarly define $C^{ss}_{\theta}, C^s_{\theta}$, and $\M(C)^{ss}_{\theta}$. One then has that the points of $\M(C)^{ss}_{\theta}$ are in one-to-one correspondence with the isomorphism classes of $\theta$-polystable modules in $C$.

In general, it is difficult to describe or construct stable modules. The lemma below, which will be used in proving Theorem \ref{thm:tamemf}, identifies modules $M$ which are stable with respect to a canonical weight associated to $M$. Specifically, for an arbitrary $A$-module $M$, we define the weight $\theta^M$ by 
\begin{equation}
\theta^M(\ddim X)=\dim_K \Hom_A(P_0,X)-\dim_K \Hom_A(P_1,X),
\end{equation}
where $P_1 {\buildrel f \over \to} P_0 \to M \to 0$ is a minimal projective presentation of $M$ in $\module(A)$ and $X$ is an $A$-module (see \cite{DomokosFFT}). Equivalently, we can write
\begin{equation}
\theta^M(\ddim X)=\dim_K \Hom_A(M,X)-\dim_K \Hom_A(X,\tau M),
\end{equation}
where $\tau M$ is the Auslander-Reiten translation of $M$ (for more details, see \cite[\S IV.2]{assemetal}).
Recall that an $A$-module $M$ is said to be \key{homogeneous} if $M \simeq \tau M$, and is said to be \key{Schur} if $\End_A(M) \simeq K$.
 
\begin{lemma}\label{lemma:stable-Schur-homogeneous} If $M$ is a homogeneous Schur $A$-module, then $M$ is $\theta^M$-stable.
\end{lemma}
\begin{proof} Since $M$ is homogeneous, we have that $\theta^{M}(\ddim M)=0$. Next, using that $M$ is also Schur, we have that for any proper $A$-submodule $0\neq M' \subset M$, $\Hom_A(M,M')=0$ and $\dim_K \Hom_A(M',\tau M)=\dim_K \Hom_A(M',M)>0$. So, $\theta^M(\ddim M')<0$ for all proper submodules $M'$ of $M$.
\end{proof}

\begin{remark} The property that a module is Schur does not guarantee the existence of a weight with respect to which the module becomes (semi-)stable (see \cite[\S~3.2]{Reineke:2008fk} for an example).
\end{remark}

\section{General results and algebras with a preprojective component}\label{sect:general}
Since we are interested in when the DO property, MF property, and representation finite properties coincide, we start by proving some general facts about these classes of algebras.  The following property provides a bridge in one direction.

\begin{definition}\label{def:schurfin}
We say that $A$ is \key{Schur-representation-finite} if, for each dimension vector $\dd$ of $A$, there are finitely many $\dd$-dimensional Schur $A$-modules up to isomorphism. 
\end{definition}
The larger class of brick-tame bocses was introduced and studied by Bodnarchuk-Drozd in \cite{MR2609188}; it includes bocses that appear in the study of vector bundles on degenerations of elliptic curves and coadjoint actions of linear groups.

\begin{lemma}\label{lem:infschur} If $A$ has the DO property, then $A$ is Schur-representation-finite.
\end{lemma}
\begin{proof} If $A$ is not Schur-representation-finite, then there must be an irreducible component of some $\module(A,\dd)$ with infinitely many Schur modules.  But there can only be one dense orbit in a component, and dimension of endomorphism rings strictly increases when moving from an orbit to its boundary, a contradiction.
\end{proof}  

\begin{prop}\label{prop:schurfin-mf} A Schur-representation-finite algebra $A$ is MF.   In particular, if an algebra has the DO property, then it also has the MF property.
\end{prop}
\begin{proof}
Let $\dd$ be a dimension vector, $C \subseteq \module(A,\dd)$ an irreducible component, and $\theta$ a weight such that $\dim_K \SI(C)_{\theta} >0$. First, we note that $\dim_K \SI(C)_{n \theta} \leq \dim_K \SI(C)_{(n+1)\theta}$ for all $n \geq 1$. Indeed, fixing a non-zero semi-invariant $f_0 \in \SI(C)_{\theta}$, the map 
\begin{equation}
\begin{split}
\SI(C)_{n \theta} &\to \SI(C)_{(n+1)\theta}\\
f &\mapsto f_0 \cdot f
\end{split}\end{equation}
is an injective linear map since $K[C]$ is a domain. The desired inequality now follows. 

Retaining the notation from Section \ref{sect:moduli}, recall that the (closed) points of $\M(C)^{ss}_{\theta}$ are in one-to-one correspondence with the $\theta$-polystable modules in $C^{ss}_{\theta}$. Furthermore, any $\theta$-polystable $A$-module is a finite direct sum of Schur modules since any $\theta$-stable module is Schur.  So the moduli space $\M(C)^{ss}_{\theta}$ is zero dimensional and since the dimensions of the graded pieces of the algebra defining $\M(C)^{ss}_{\theta}$ weakly increase, we conclude that $\dim_K \SI(C)_{\theta}=1$. 
\end{proof}




Now we are almost ready to prove our first main result.  The following lemma allows us to reduce to the minimal representation-infinite case.

\begin{lemma}\label{lem:mfquot}
Any quotient of an MF bound quiver algebra is MF.
\end{lemma}
\begin{proof}
Let $A$ be a MF algebra, $I$ an ideal of $A$, and $\dd$ a dimension vector of $A$. Then, any irreducible component $C \subseteq \module(A/I,\dd)$ is embedded ($\GL(\dd)$-equivariantly) in an irreducible component $C' \subseteq \module(A,\dd)$. We know from invariant theory \cite[Corollary 2.2.9]{MR1918599} that the above embedding gives rise to a surjective map at the level of $\SL(\dd)$-invariant rings which preserves weight spaces.  So for any weight $\theta$, we have that $\dim_K \SI(C)_{\theta} \leq \dim_K \SI(C')_{\theta} \leq 1$.
\end{proof}

Recall that an algebra $A$ is said to \emph{admit a preprojective component} if its Auslander-Reiten quiver has an acyclic connected component in which every indecomposable is of the form $\tau^{-n} P$ for some projective $P$.  For example, hereditary algebras $A=KQ$ always admit a preprojective component.  See \cite[\S~VIII.2]{assemetal} for more details (where the terminology ``postprojective'' is used.)

\begin{theorem}\label{thm-preproj-comp} Let $A$ be a connected, bound quiver algebra with a prepojective component. Then, the following properties are equivalent:
\begin{enumerate}
\renewcommand{\theenumi}{\arabic{enumi}}
\item $A$ is representation-finite;

\item for each dimension vector $\dd$ of $A$, the group $\GL(\dd)$ acts on each irreducible component of $\module(A,\dd)$ with a dense orbit;

\item for each dimension vector $\dd$ of $A$ and each irreducible component $C$ of $\module(A,\dd)$, the algebra of semi-invariants $K[C]^{\SL(\dd)}$ is multiplicity-free. 
\end{enumerate}
\end{theorem}

\begin{proof}We have seen that the implications $(1)\Longrightarrow (2) \Longrightarrow (3)$ hold true for arbitrary bound quiver algebras. 

It remains to prove the implication $(3) \Longrightarrow (1)$. Assume to the contrary that the MF algebra $A$ is representation-infinite. It follows from the work of Happel and Vossieck that any connected algebra admitting a prepojective component has a tame concealed algebra as a quotient (see \cite{MR701205} or \cite[Theorem XIV.3.1]{MR2360503}).  This result combined with Lemma \ref{lem:mfquot} tells us that $A$ has a quotient $B$ which is an MF tame concealed algebra. Denote by $\hh$ the dimension vector of an indecomposable $B$-module lying at the mouth of a homogeneous tube of $B$. It is well-known that $\module(B,\hh)$ is irreducible and that there is always an integral weight $\theta$ of $B$ such that $\M(B,\hh)^{ss}_{\theta} \simeq \PP^1$ (see for example \cite{MR3037894}). In particular, $\SI(B,\hh)$ is not multiplicty-free, a contradiction.
\end{proof}


\section{Representation-infinite DO algebras}\label{sect:ctreg}
\subsection{Example of a representation-infinite DO algebra}
Our example of a representation-infinite DO algebra works in arbitrary characteristic and is given by the following quiver with relations.
\begin{equation}\label{eq:DO}
\vcenter{\hbox{
\begin{tikzpicture}[point/.style={shape=circle,fill=black,scale=.5pt,outer sep=3pt},>=latex]
\node[point,label={below:$1$}] (1) at (0,0) {} ;
\node[point,label={below:$2$}] (2) at (2,0) {} edge[in=45,out=-45,loop] node[right] {$b$} ();
\path[->] (1) edge node[above] {$a$} (2) ;
\end{tikzpicture} }}
\qquad
b^n = b^2 a = 0, n \in \NN
\end{equation}
We can think of a representation of $\Lambda$ as a nilpotent operator on a vector space, along with a distinguished subspace. For $n=6$, the algebra is tame representation-infinite with distributive ideal lattice (see, for example \cite[p.~242]{MR799266}). For $n > 6$, the algebra is wild \cite{MR949903}. Fix $n$ and denote this algebra by $\Lambda$ throughout this section.

\begin{theorem}\label{thm-example-DO} 
The algebra $\Lambda$ is DO for all $n$.
\end{theorem}

The proof proceeds by induction on the dimension vector.  Using Remark \ref{do-components}, it is enough to take a general element of an irreducible component and show that it has a direct summand whose orbit is dense.  We do this by direct calculation, so that our method is actually an algorithm to calculate the generic module in a given component, yielding a classification of indecomposable components.  A ``conceptual'' proof without calculations does not seem to be within reach using current technology (see Section \ref{sect:doext}).  But we can make the computations easier by reframing the problem in terms of matrices over the polynomial ring $K[x]$.

Let $\dd =(d_1,d_2)$ be a dimension vector for $\Lambda$, and fix an irreducible component $C$ of $\module(\Lambda, \dd)$.  Denote by $A, B$ matrices representing the actions of $a,b$ at a typical point of $\module(\Lambda, \dd)$.
There is an open subset of $C$ on which the Jordan type of $B$ is constant.  In this case we say that $C$ is of type $\lambda$, where $\lambda = (\lambda_1, \lambda_2, \dotsc)$ is a partition such that $\lambda_i = \dim \ker B^i$.  
Now we fix a representative for $B$ of general Jordan type in $C$, and let $Z_B \subseteq \GL(d_2)$ be the centralizer of $B$. Then we see that $C$ has a dense $\GL(\dd)$ orbit if and only if $\Hom_K (V_1, V_2)\simeq \Mat_{d_2\times d_1}$ has a dense $H:=\GL(d_1) \times Z_B$ orbit. 

We identify the pair $(K^{d_2}, B)$ with a torsion $K[x]$-module $X$, and fix a decomposition into indecomposables
\begin{equation}
X \simeq \bigoplus_{i=1}^n {\bar{\lambda}_i}\, J_i,
\end{equation}
where $J_i := K[x]/(x^i)$ and $\bar{\lambda}_i:= \lambda_i - \lambda_{i+1}$ is the multiplicity of this summand. 
The centralizer of $B$ in $\GL(d_2)$ is naturally identified with $\Aut_{K[x]}(X)$ in this way, and the matrix $A$ can be thought of as having entries in $K[x]$.  Each row of $A$ corresponds to a summand $J_i$, which we order by increasing dimension.  So if we group the summands into isotypical components $X_i := {\bar{\lambda}_i}\, J_i$, the matrix $A$ has block form 

\begin{equation}
\vcenter{\hbox{
\begin{tikzpicture}
\matrix (m) [matrix of math nodes,nodes in empty cells,left delimiter  = {[}, right delimiter = {]} ]
{
\phantom{X_1} &  &\phantom{X_1}\\
\phantom{X_1} &  &\phantom{X_1}\\
\phantom{\vdots} &\vdots  &\phantom{X_1}\\
\phantom{X_1} &  &\phantom{X_1}\\
\phantom{X_1} &  &\phantom{X_1}\\
};
\draw[thick] (m-1-1.south west) -- (m-1-3.south east);
\draw[thick] (m-2-1.south west) -- (m-2-3.south east);
\draw[thick] (m-4-1.south west) -- (m-4-3.south east);
\node [left=of m-1-1.east] {$X_1$};
\node [left=of m-2-1.east] {$X_2$};
\node [left=of m-3-1.east] {$\vdots$};
\node [left=of m-5-1.east] {$X_n$};
\end{tikzpicture}}},
\end{equation}
and the condition $b^2a =0$ is equivalent to requiring each entry in block row $X_i$ to be in the ideal $(x^{i-2})$, that is, of the form $ax^{i-2} + bx^{i-1}$ for some sufficiently general $a, b \in K$.  The group $H$ acts by:
\begin{itemize}
\item arbitrary $K$-linear column operations;
\item row operations of the form ``add $f$ times a row in block $X_i$ to a row in block $X_j$,'' where $f \in K[x]$ for $i \leq j$, and $f \in (x^{i-j})$ for $i > j$.
\end{itemize}
In this notation, we denote a representation of $\Lambda$ by a pair $(W, M)$ where $M$ is a $K[x]$ module and $W \subseteq M$ is a $K$-linear subspace.

Now we must analyze several cases, branching on the relation between $d_1$ and $\lambda$.  In what follows, we let $\mathbf{1}_{d}$ denote a $d\times d$ identity matrix and $*x^i$ a region of entries in the ideal $(x^i)$ (or simply $*$ when $i=0$).  We use the fact that we are working with a sufficiently generic matrix in our component throughout, without explicit mention.  By induction we can also assume $\bar{\lambda}_n \neq 0$, otherwise $C$ would be a component for an algebra of smaller dimension.
For the sake of readability we display matrices for the case $n=5$.

Using Gauss elimination, we can reduce to the case $d_1 \geq \lambda_2$ and put $A$ into the form
\begin{equation}\label{eq:step2}
A=\vcenter{\hbox{
\begin{tikzpicture}
\matrix (m) [matrix of math nodes,nodes in empty cells,left delimiter  = {[}, right delimiter = {]},
nodes={minimum width=2em} ]
{
* & * &* & * & * \\
{\mathbf{1}}_{\bar{\lambda}_2} & * &* & * & *  \\
0 & x \mathbf{1}_{\bar{\lambda}_3}  & *x & *x & *x \\
0 & 0  &x^2 \mathbf{1}_{\bar{\lambda}_4} & *x^2 & *x^2\\
0 & 0 & 0 & x^3 \mathbf{1}_{\bar{\lambda}_5} & *x^3  \\
};
\node [left=of m-1-1.east] {$X_1$};
\node [left=of m-2-1.east] {$X_2$};
\node [left=of m-3-1.east] {$X_3$};
\node [left=of m-4-1.east] {$X_4$};
\node [left=of m-5-1.east] {$X_5$};
\end{tikzpicture}}}
\end{equation}
where the last column is of width 0 if $d_1 = \lambda_2$.  (If $d_1 < \lambda_2$ then there would be a row of 0s at bottom, giving a direct summand $(0, J_n)$ of a general representation.)  Then using column operations to the right, we can clear leading terms to increase powers of $x$.
\begin{equation}\label{eq:step3}
A=\vcenter{\hbox{
\begin{tikzpicture}
\matrix (m) [matrix of math nodes,nodes in empty cells,left delimiter  = {[}, right delimiter = {]},
nodes={minimum width=2em} ]
{
* & * &* & * & * \\
{\mathbf{1}}_{\bar{\lambda}_2} & *x &*x & *x & *x  \\
0 & x \mathbf{1}_{\bar{\lambda}_3}  & *x^2 & *x^2 & *x^2  \\
0 & 0  &x^2 \mathbf{1}_{\bar{\lambda}_4} & *x^3 & *x^3 \\
0 & 0 & 0 & x^3 \mathbf{1}_{\bar{\lambda}_5} & *x^4  \\
};
\node [left=of m-1-1.east] {$X_1$};
\node [left=of m-2-1.east] {$X_2$};
\node [left=of m-3-1.east] {$X_3$};
\node [left=of m-4-1.east] {$X_4$};
\node [left=of m-5-1.east] {$X_5$};
\end{tikzpicture}}}.
\end{equation}

We claim that the top row always splits off, taking the bottom row with it precisely when $d_1 = \lambda_2$.  This is seen by direct computation: take the upper right entry and use column operations to clear entries of the same degree in its row.  This first "messes up" the lower part, but these blocks can be restored by Gaussian elimination again (from strictly lower rows).
Now the top row has only one nonzero entry, so it can almost clear the column below it since row operations to lower blocks contribute factors of $x$.  The only way the column fails to be completely cleared is if $d_1 = \lambda_2$, in which case the entry is $x^{n-2}$ instead of $x^{n-1}$.  

In summary, if we let $m$ be such that the top row is of type $J_m$ (i.e., $m$ is minimal such that $\bar{\lambda}_m \neq 0$), then the direct summands that split off are as follows.  If $d_1 > \lambda_2$, we get the direct summand $(Kx^{m-2} + Kx^{m-1}, J_m)$, and if $d_1 = \lambda_2$ then we get the direct summand $(K(x^{m-2}\oplus 0) + K(x^{m-1} \oplus x^{n-2}), J_m\oplus J_n)$. For $m=1$ we take $x^{m-2} = 0$, and in case $m=n$ the later type further decomposes.   We summarize the classification of indecomposable components obtained from this algorithm.
\begin{equation}\label{eq:indecomp}
\begin{array}{|c|c|c|}
\hline
\dd & \text{Jordan type} &  A\\
\hline
(1,0) & \text{none} & \text{none}\\
(0, m) & J_m & \text{none} \\
(1, 1) & J_1 & [1] \\
(1, m) &  J_m,\ m \geq 2 & [x^{m-1}]\\
(2, m) &  J_m,\ m \geq 2 & [x^{m-1}\ x^{m-2}] \\
(1, 1+n) & J_1 \oplus J_n & \begin{bmatrix}1\\x^{n-2}\end{bmatrix}   \\
(2, m+n) & J_m \oplus J_n ,\ m < n & \begin{bmatrix}x^{m-2} & x^{m-1}\\ 0 & x^{n-2}\end{bmatrix}  \\
\hline
\end{array}
\end{equation}

\subsection{Dense orbit conjecture and open questions}

The following conjecture of Weyman has guided our results connecting the DO property with representation theory.  In this subsection we propose a proof strategy and carry out some steps towards the conjecture.

\begin{conjecture}[{\bf Dense Orbit Conjecture}]\label{conj:repfiniteDO} For a triangular algebra $A$, the following statements are equivalent:
\begin{enumerate}
\renewcommand{\theenumi}{\roman{enumi}}
\item $A$ is representation-finite;
\item for any ideal $I$ of $A$, the algebra $A/I$ is DO.
\end{enumerate}
\end{conjecture}

The formulation of statement (ii) allows us to reduce to the case that $A$ is minimal representation-infinite, where we can use recent work of Bongartz and Ringel.  They have shown that any minimal representation-infinite, finite-dimensional algebra is in one of the following three categories (see the Introduction of \cite{MR2931904}):
\begin{enumerate}
\item Algebras with a non-distributive ideal lattice;
\item Algebras with a ``good'' universal cover that contains a convex subcategory which is tame concealed of type $\widetilde{D}_n, \widetilde{E}_6, \widetilde{E}_7$, or $\widetilde{E}_8$;
\item Algebras with a ``good'' universal cover, but in which all finite convex subcategories are representation finite.
\end{enumerate}
In fact, Ringel showed that all algebras in case (3) are string algebras, and explicitly classified them in terms of quivers with relations.
We now prove the conjecture for classes (1) and (3).

The following proposition, covering case (1), comes directly by following Bongartz's argument in \cite[\S~1]{MR3038490}, with the additional assumption that $A$ is triangular.

\begin{prop}\label{prop:nondistrib-do} Let $A$ be a non-distributive triangular algebra.  Then $A$ admits infinitely many Schur representations of the same dimension, and thus $A$ is not DO.
\end{prop}
\begin{proof}
Since $A$ is non-distributive, it admits primitive idempotents $e,f$ (not necessarily distinct) such that $fAe$ is neither cyclic as a left $fAf$-module nor right $eAe$-module; in other words, the radical filtration $(R^i)$ of $fAe$ as an $eAe$-$fAf$-bimodule admits a step $R^l / R^{l+1}$ of dimension $\geq 2$.  Choose some $v,w \in R^l$ which are linearly independent modulo $R^{l+1}$.  Then without loss of generality we can mod out (the two-sided ideal generated by) $R^{l+1}, Jv, vJ, Jw, wJ$ where $J=\rad A$.

If $A$ is triangular, then in particular there are no loops at any vertex of its quiver $Q$, so that $eAe \simeq fAf \simeq K$ and $e \neq f$.  Consider the family $V_\lambda = Ae / \langle v - \lambda w \rangle$ for $\lambda \in K$.
Applying $\Hom_A( - , V_\lambda)$ to the quotient, we get $0 \to \End_A (V_\lambda) \to \Hom_A(Ae, V_\lambda)$.  But $Ae$ is projective, so $\Hom_A(Ae, V_\lambda) = eV_\lambda = eAe =K$.  Thus $V_\lambda$ is a Schur $A$-module, and $A$ is not DO by Lemma \ref{lem:infschur}.
\end{proof}

The next proposition covers case (3).  See Section \ref{sec:tameMF} below for the definition of a string algebra.

\begin{prop}\label{string-not-DO} Let $A$ be a representation-infinite string algebra. Then $A$ does not have the DO property.
\end{prop}
\begin{proof} Since $A$ is representation-infinite, we know that there are indecomposable band modules. Among the dimension vectors $\dd$ of $A$ for which there are infinitely many $\dd$-dimensional indecomposable band modules, we choose a $\dd$ with $\sum_{i\in Q_0}\dd(i)$ minimal.

Assume to the contrary that $A$ is DO. Let $C \subseteq \module(A,\dd)$ be an irreducible component that contains infinitely many indecomposable band modules. Since $A$ is DO, there exists $M_0 \in C$ such that $C=\overline{\GL(\dd)M_0}$.  We claim that $M_0$ is an indecomposable band module.
Given this, we have a contradiction: the orbits of all other band modules in the same family as $M_0$ lie in the same irreducible component, since they are obtained as the image of a certain morphism $\phi\colon \GL(\dd) \times K^* \to \module(A, \dd)$, so they are all in the boundary of the orbit $\GL(\dd)M_0$.
But the endomorphism algebras of band modules in the same family have the same dimension \cite{MR1090218}, so the orbit of one cannot be in the boundary of another.

To prove the claim, denote by $s$  the number of string modules occurring in a direct sum decomposition of $M_0$ into indecomposables.
For any indecomposable band module $M \in C$ we have
\begin{equation}
\dim_K M=\sum_{a \in Q_1} \rk M(a) \leq \sum_{a \in Q_1} \rk M_0(a)=\dim_K M_0-s,
\end{equation}
which implies that $s=0$. Now $M_0$ is an indecomposable by minimality of $\dd$.
\end{proof}

It would be nice to understand how the DO property behaves under quotients (i.e., we would like an analogue of Lemma \ref{lem:mfquot}).  So we pose the following question.

\begin{question}
Is it true that any quotient of a DO algebra is DO?
\end{question}

If the answer to the question is `yes' then statement (ii) of the Dense Orbit Conjecture can be simplified to ``$A$ is DO.''

\section{Representation-infinite MF algebras}\label{sec:rep-inf-MF-algebras}

\subsection{Tame MF algebras}\label{sec:tameMF}
In this section, we give a representation theoretic characterization of the MF property for tame algebras, and apply this to classify MF string algebras.


\begin{theorem}\label{thm:tamemf} A tame algebra is Schur-representation-finite if and only if it is MF.
\end{theorem}

\begin{proof} 
The implication ``$\Longrightarrow$'' was proved in Proposition \ref{prop:schurfin-mf} for all algebras.

Now, let $A$ be a tame algebra with the MF property. Assume for a contradiction that there is a dimension vector $\dd$ of $A$ and an irreducible component $C \subseteq \module(A,\dd)$ that contains infinitely many Schur $A$-modules. We immediately deduce from this that $C$ is not an orbit closure (see for example Lemma \ref{lem:infschur}).

From Crawley-Boevey's Theorem D in \cite{MR931510}, we know that all, except finitely many, Schur modules in $C$ are homogeneous. So let $M \in C$ be a homogeneous Schur $A$-module. Then the $\theta^M$-stable locus $C^s_{\theta^M} \subset C$ is non-empty and dense by Lemma \ref{lemma:stable-Schur-homogeneous}, and moreover, $C^s_{\theta^M}$ must be an infinite disjoint union of orbits since otherwise $C$ would be an orbit closure. Consequently, we have that $\dim \M(C)^{ss}_{\theta^M} \geq 1$, which is in contradiction to $A$ being MF.
\end{proof}

String algebras are a well-understood class of tame algebras whose indecomposable representations can be completely described in terms of certain paths in the associated quiver. Our characterization of the MF property for tame algebras can be made even more explicit for string algebras.

A bound quiver algebra $A=KQ/I$ is said to be a \key{string algebra} if $I$ can be generated by a set of relations $\mathcal{R}$ satisfying the  following conditions:
\begin{enumerate}
\renewcommand{\theenumi}{\arabic{enumi}}
\item each vertex of $Q$ is the tail of at most two arrows, and the head of at most two arrows;

\item each relation in $\mathcal{R}$ is just a monomial in the arrows of $Q$;

\item for each arrow $b \in Q_1$, there is at most one arrow $a \in Q_1$ with $ta=hb$ and at most one arrow $c \in Q_1$ with $tb=hc$ such that $ab \notin \mathcal{R}$ and $bc \notin \mathcal{R}$. 
\end{enumerate}
For the details behind the construction of their indecomposable representations, known as string and band modules, we refer the reader to \cite{MR876976,MR801283}.

\begin{corollary}\label{cor:MFstring}
A string algebra $KQ/I$ is MF if and only if every subquiver $L \subseteq Q$ of type $\widetilde{A}_n$ contains a relation from $I$.
\end{corollary}
\begin{proof}
If $Q$ contains a type $\tilde{A}_n$ subquiver $L$ in which there is no relation, then the family of band modules which are supported precisely on $L$ and one dimensional at each vertex shows that $KQ/I$ is not Schur-representation-finite and thus not MF.

On the other hand, if every subquiver of type $\tilde{A}_n$ in $Q$ contains a relation, then any band must traverse some vertex more than once. 
Krause's computation of morphisms between band modules (see conditions (H1)-(H4) of \cite[p.193]{MR1090218}) then shows that each band module admits a nilpotent endomorphism.  Thus $KQ/I$ has no Schur band modules and must be Schur-representation-finite, so it is also MF.
\end{proof}

The following example illustrates this corollary, and was in fact our first example of a representation-infinite MF algebra.

\begin{example}\label{ex:butterly}
Consider the string algebra $\Lambda=KQ/I$ given by the ``butterfly quiver''
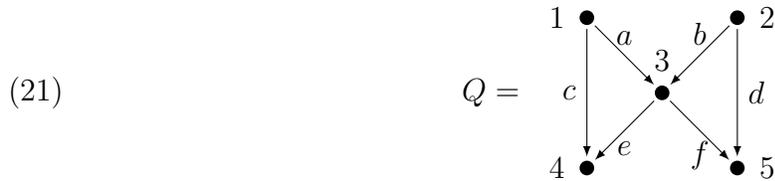
\begin{equation}
Q=
\vcenter{\hbox{  
 \begin{tikzpicture}[point/.style={shape=circle,fill=black,scale=.5pt,outer sep=3pt},>=latex]
   \node[point,label={left:$1$}] (1) at (-1,1) {};
  \node[point,label={right:$2$}] (2) at (1,1) {};
  \node[point,label={above:$3$}] (3) at (0,0) {};
  \node[point,label={left:$4$}] (4) at (-1,-1) {};
  \node[point,label={right:$5$}] (5) at (1,-1) {};
 
  \path[->]
  	(1) edge node[above] {$a$} (3)
  	(1) edge node[left] {$c$} (4)
  	(2) edge node[above] {$b$} (3)
  	(2) edge node[right] {$d$} (5)
  	(3) edge node[below] {$e$} (4)
  	(3) edge node[below] {$f$} (5);
   \end{tikzpicture} 
   }} 
\end{equation}
with $I$ the ideal generated by all length 2 paths: $ea, eb, fa, fb$.  We see that
$\Lambda$ is a string algebra with 
\begin{equation}
B=c^{-1}ef^{-1}db^{-1}a
\end{equation}
the only primitive band, and in fact $\Lambda$ is a minimal representation-infinite algebra. The only subquivers of type $\tilde{A}_n$ are the left and right triangles of the diagram, which each contain a length two path in the ideal of relations.  So by the corollary, $\Lambda$ is MF.
\end{example}

\subsection{MF conjecture and open questions}
Finally, we state another conjecture of Weyman, which is that Theorem \ref{thm:tamemf} holds without the assumption that $A$ is tame.

\begin{conjecture}[{\bf Multiplicity Free Conjecture}]\label{conj:MFC}
An algebra is Schur-representation-finite if and only if it is MF.
\end{conjecture}

It has been pointed out to us by Geiss and Schr\"oer that preprojective algebras of Dynkin quivers, which are generally wild with infinite global dimension, are Schur-representation-finite because every Schur module is rigid.

\begin{question}
Do there exist wild, Schur-representation-finite algebras of finite global dimension?
\end{question}


\bibliographystyle{alpha}
\bibliography{dense-orbits}

\end{document}